\documentclass{amsart}
\usepackage{amsfonts}
\usepackage{amsmath,amssymb}
\usepackage{amsthm}
\usepackage{amscd}
\usepackage{enumerate}
\usepackage{graphics}
\usepackage{graphicx}
\theoremstyle{remark}{
\newtheorem{Def}{{\rm Definition}}
\newtheorem{Ex}{{\rm Example}}
\newtheorem{Rem}{{\rm Remark}}

\newtheorem{MainProb}{Main Problem}
}
\theoremstyle{plain}{

\newtheorem{Prop}{Proposition}

\newtheorem{MainThm}{Main Theorem}

\newtheorem{Fact}{Fact}
}

\begin{document}

\title[Simple polyhedra regarded as the Reeb spaces of stable fold maps]{Simple polyhedra homeomorphic to the Reeb spaces of stable fold maps}
\author{Naoki Kitazawa}
\address{Institute of Mathematics for Industry, Kyushu University, 744 Motooka, Nishi-ku Fukuoka 819-0395, Japan\\
 TEL (Office): +81-92-802-4402 \\
 FAX (Office): +81-92-802-4405 \\
}
\email{n-kitazawa@imi.kyushu-u.ac.jp}
\urladdr{https://naokikitazawa.github.io/NaokiKitazawa.html}
\maketitle
\begin{abstract}
{\it Simple polyhedra} are $2$-dimensional polyhedra and important objects in low-dimensional geometry
and in the applications of {\it fold} maps, defined as smooth maps regarded as higher dimensional variants of Morse functions. For example, they are locally so-called {\it Reeb spaces} of (so-called {\it stable}) fold maps into the plane and represent the manifolds compactly. The {\it Reeb space} of a {\it fold} map is the space of all connected components of preimages of it and is a polyhedron whose dimension is same as that of the manifold of the target.

Is a given simple polyhedron homeomorphic to the Reeb space of a suitable stable fold map? What are their global topologies like? Previously the author has challenged this for a specific case and presented fundamental construction and topological properties of the polyhedra as new results. The present paper extends some of these works and results and present results of new types.



\end{abstract}


\maketitle
\section{Introduction.}
\label{sec:1}

{\it Simple polyhedra} are $2$-dimensional polyhedra which play important roles in various fields of $3$ or $4$-dimensional geometry. \cite{ikeda} is one of pioneering studies related to this. Before introducing the notion, we introduce some fundamental terminologies and notions and some of notation.

For a set $X$, $\sharp X$ denotes the size or cardinality of $X$.
For a smooth manifold $X$ and a point $p \in X$, $T_pX$ denotes the tangent vector space at $p$.
For a smooth map $c:X \rightarrow Y$, ${dc}_p:T_pX \rightarrow T_{c(p)} Y$ denotes the differential of $c$ at $p$. For a smooth manifold $X$, $\dim X$ denotes its dimension.
 

\begin{Def}
\label{def:1}
A family $\{c_j:X_j \rightarrow Y\}_{j \in J}$ of finitely many smooth immersions each of which is a smooth immersion from a closed manifold into a manifold with no boundary is said to {\it have normal crossings only as crossings} if the following conditions are satisfied. We also define this notion for a single immersion naturally.

\begin{enumerate}

\item For each $q \in Y$, the union ${\bigcup}_{j \in J} {c_j}^{-1}(q)$ is a finite set.
\item Let ${c_j}^{-1}(q)$ be denoted by $J^{\prime}:=\{q_{j,j^{\prime}} \mid 1 \leq j^{\prime} \leq \sharp {c_j}^{-1}(q)\}$.  
${\rm rank}\ ({\bigcap}_{(j,j^{\prime}) \in J \times j^{\prime}}  {dc_j}_{q_{j,j^{\prime}}} (T_{q_{j,j^{\prime}}}X_j))+{\Sigma}_{j \in J}  (\dim Y-\dim X_j) \sharp {c_j}^{-1}(q)=\dim Y$.
\end{enumerate} 

\end{Def}

Hereafter, we need fundamental notions, properties, and principles on the PL category and the piecewise smooth category, which is regarded as a category equivalent to the PL category. See \cite{hudson} and see also \cite{bryant} for systematic expositions.

{\it Polyhedra} are objects in these categories and we can define their {\it PL structures} uniquely. They are said to be {\it PL homeomorphic} if there exists a homeomorphism which is a PL or piesewise smooth map. 
{\it PL} manifolds are polyhedra homeomorphic to topological manifolds whose PL structures satisfy some natural conditions.
We can define {\it PL} ({\it piesewise smooth}) {\it emmbeddings} and {\it subplolyhedra} for example. Smooth manifolds are PL manifolds in canonical ways. For a smooth manifold $X$, or more generally, a polyhedron $X$, $\dim X$ denotes its dimension.

We mainly encounter $1$ or $2$-dimensional polyhedra as objects in these categories. For topological spaces which are homeomorphic to these polyhedra PL structures are known to be unique. Note also that $1$, $2$ and $3$-dimensional topological manifolds have the structures of PL and smooth manifolds uniquely. See \cite{moise} for example.

We present several important notions and theory which are important in our study and related studies. It is important that every continuous map between polyhedra is regarded as a PL (piecewise smooth) map by considering a suitable approximation by a homotopy. We present \cite{hirsch} as a paper on {\it regular neighborhoods} of a subpolyhedron discussed in the smooth category. Such facts and notions will also appear implicitly throughout the present paper.

\begin{Def}
\label{def:2}
If a polyhedron $P$ is PL homeomorphic to the image of a smooth immersion $i_P$, represented as the disjoint union of all smooth immersions in some family $\{c_{i}:X_i \rightarrow Y\}_{i \in I}$ and the family has normal crossings only as crossings, then the immersion $i_P$ and this family are said to be {\it premaps} for $P$. We also call them {\it premaps} if the spaces of the targets are restricted to make the maps surjections.
\end{Def}

Hereafter, ${\mathbb{R}}^k$ denotes the $k$-dimensional Euclidean space for each integer $k \geq 0$ and ${\mathbb{R}}^1$ is also denoted by $\mathbb{R}$ in a natural way. $\mathbb{Z} \subset \mathbb{R}$ denotes the ring of all integers. We regard ${\mathbb{R}}^k$ as the smooth manifold with the natural differentiable structure and the Riemannian manifold with the standard Euclidean metric. $||x|| \geq 0$ denotes the distance between $x \in {\mathbb{R}}^k$ and the origin $0 \in {\mathbb{R}}^k$. 

A {\it diffeomorphism} is a smooth homeomorphism with no singular points or such a homeomorphism of the class $C^{\infty}$. This is also a smooth embedding.
A {\it smooth} isotopy $\Phi:X \times [0,1] \rightarrow Y$ is a smooth map which is an isotopy such that ${\Phi} {\mid}_{X \times \{t\}}$ is a smooth embedding for each $t \in [0,1]$. If one of the embedding here is a diffeomorphism, then all smooth embeddings here are diffeomorphisms. A {\it PL} ({\it smooth}) bundle means a bundle whose fiber is a polyhedron (resp. smooth manifold) and whose structure group consists of 
piesewise smooth homeomorphisms (resp. diffeomorphisms).

\begin{Def}
\label{def:3}
A {\it simple polyhedron} $P$ is a $2$-dimensional compact polyhedron satisfying the following three.
\begin{enumerate}
\item \label{def:3.1}
There exists a family $\{C_i\}_{i \in I}$ of finitely many $1$-dimensional mutually disjoint subpolyhedra of $P$: we call the disjoint union of all these subpolyhedra the {\it branch} of $P$. 
\item \label{def:3.2}
$P-{\sqcup}_{i \in I} C_i$ is a $2$-dimensional smooth manifold with no boundary.
\item \label{def:3.3}
Each $C_i$ and a small regular neighborhood $N(C_i)$ satisfy either of the following two.
\begin{enumerate}
\item \label{def:3.3.1}
$C_i$ is a circle and $N(C_i)$ is the total space of a trivial PL bundle over $C_i$ whose fiber is a closed interval where $C_i \subset N(C_i) \subset P$ is identified with
one connected component of the boundary: $C_i$ is identified with $C_i \times \{0\} \subset C_i \times [0,1]$.
\item \label{def:3.3.2}
A premap $\{c_{i,j}:C_{i,j} \rightarrow {\mathbb{R}}^2\}_{j \in J_i}$ for the polyhedron $C_i$ exists and the following {\rm (}PL{\rm )} topological properties are enjoyed.
\begin{enumerate}
\item \label{def:3.3.2.1}
$N(C_i)$ is the quotient space of the total space $N({\sqcup}_{j \in J_i} C_{i,j})$ of a PL bundle over ${\sqcup}_{j \in J_i} C_{i,j}$ whose fiber is PL homeomorphic to $K:=K_0 \bigcup K_1 \bigcup K_2$ with $K_j:=\{(r\cos \frac{2}{3}j\pi ,r\sin \frac{2}{3}j\pi)\mid 0 \leq r \leq 1.\}$ and whose structure group is a trivial group or a group of order $2$ acting in the following way.
\begin{enumerate}
\item The action by the non-trivial group fixes (all points in) $K_0$.
\item The non-trivial element maps $(r\cos \frac{2}{3}\pi, r\sin \frac{2}{3}\pi)$ to $(r\cos \frac{4}{3}\pi, r\sin \frac{4}{3}\pi)$. 
\end{enumerate}
${\sqcup}_{j \in J_i} C_{i,j} \subset N({\sqcup}_{j \in J_i} C_{i,j})$ is identified with the total space of the subbundle ${\sqcup}_{j \in J_i} C_{i,j} \times \{0\} \subset N({\sqcup}_{j \in J_i} C_{i,j})$, obtained by restricting the fiber to $\{0\} \subset K \subset {\mathbb{R}}^2$.
\item \label{def:3.3.2.2}
$N(C_i)$ is obtained from $N({\sqcup}_{j \in J_i} C_{i,j})$ as the quotient space of $N({\sqcup}_{j \in J_i} C_{i,j})$ as follows. We identify $C_i$ with the image of the premap ${\sqcup}_{j \in J_i} c_{i,j}:{\sqcup}_{j \in J_i} C_{i,j} \rightarrow {\mathbb{R}}^2$. The preimage of a point in the image of this premap consists of at most two points and the number of points whose preimages consist of exactly two points is finite. Let $\{p_a\}_{a \in A} \subset C_i$ denote the set of all of such points: we call such a point a {\it vertex} of $P$. We consider a suitable small regular neighborhood of a point $p_{a,1}$ in the preimage $\{p_{a,1},p_{a,2}\}$ of $p_a$, denoted by a closed interval $[l_{a,1},u_{a,1}]$ and we consider a trivialization $[l_{a,1},u_{a,1}] \times K$ of the restriction of the bundle $N({\sqcup}_{j \in J_i} C_{i,j})$ over ${\sqcup}_{j \in J_i} C_{i,j}$ there. For the point $p_{a,2}$, we consider a suitable small regular neighborhood, denoted by a closed interval $[l_{a,2},u_{a,2}]$ and consider a trivialization $[l_{a,2},u_{a,2}] \times K$ of the restriction of the bundle $N({\sqcup}_{j \in J_i} C_{i,j})$ over ${\sqcup}_{j \in J_i} C_{i,j}$ there. We identify $[l_{a,1},u_{a,1}] \times {(K_0 \bigcup K_{b_1})}$ for $b_1=1$ or $b_1=2$ with $[l_{a,2},u_{a,2}] \times {(K_0 \bigcup K_{b_2})}$ for $b_2=1$ or $b_2=2$. These two spaces are regarded as products of closed intervals. We identify them via a piesewise smooth homeomorphism mapping $\{\frac{l_{a,1}+u_{a,1}}{2}\} \times (K_0 \bigcup K_{b_1})$  onto $[l_{a,2},u_{a,2}] \times \{0\}$. By such an identification for each $a$, we have $N(C_i)$.
\end{enumerate}
\end{enumerate}
\end{enumerate}
\end{Def}
One of important studies on this is \cite{ikeda}, studying global topologies of simple polyhedra and also presenting important topics in low-dimensional topology. \cite{turaev} is also important. Note that the name "simple polyhedron" may not be used for such a polyhedron in general. {\it Branched surfaces} (in \cite{kitazawa5}) are simple polyhedra without vertices. {\it Surfaces} are regarded as $2$-dimensional manifolds which may not be compact or closed.

For a continuous map $c:X \rightarrow Y$, an equivalence relation ${\sim}_c$ on $X$ is defined: $x_1 {\sim}_c x_2$ if and only if $x_1$ and $x_2$ are in a same connected component of a preimage $c^{-1}(y)$ for some $y \in Y$.
\begin{Def}
\label{def:4}
The quotient space $W_c:=X/{\sim}_c$ is the {\it Reeb space} of $c$.
\end{Def}
$q_c:X \rightarrow W_c$ denotes the quotient map and we can define a map $\bar{c}$ in a unique way by the relation $c=\bar{c} \circ q_c$. One of pioneering articles on Reeb spaces is \cite{reeb}. 

For a smooth map $c:X \rightarrow Y$, a {\it singular} point $p \in X$ is defined as a point such that the rank of the differential ${dc}_p$ there is smaller than $\min\{\dim X,\dim Y\}$. We call the set of all singular points of $c$ the {\it singular set} of $c$, denoted by $S(c)$. $c(S(c))$ is called the {\it singular value set} of $c$. The complementary set of the singular set of $c$ is the {\it regular value set} of $c$. 

\begin{Def}
\label{def:5}
Let $M$ be an $m$-dimensional closed and smooth manifold and $N$ an $n$-dimensional smooth manifold with no boundary with $m \geq n \geq 1$. A smooth map $f:M \rightarrow N$ is said to be a {\it fold} map if at each singular point $p$ $f$ is represented by the form $(x_1,\cdots,x_m) \rightarrow (x_1,\cdots,x_{n-1},{\Sigma}_{j=n}^{m-i(p)} {x_j}^2-{\Sigma}_{m-i(p)+1}^{m} {x_j}^2)$
for suitable coordinates and a suitable integer $0 \leq i(p) \leq \frac{m-n+1}{2}$.
\end{Def}
\begin{Prop}
\label{prop:1}
In Definition \ref{def:5}, for any singular point $p$ of $f$, $i(p)$ is unique and we can define the {\rm index} of $p$ by the integer $i(p)$. 
The set of all singular points of a fixed index is a closed and smooth submanifold with no boundary and the dimension is $n-1$. The restriction of $f$ to this submanifold is a smooth immersion.
\end{Prop}
\begin{Def}
\label{def:6}
In Definition \ref{def:5}, if the restriction $f {\mid}_{S(f)}$ to the singular set has normal crossings only as crossings, then $f$ is said to be {\it stable}.  
\end{Def}
It is well-known that by a slight perturbation, a fold map is deformed to a stable one where the topology on the space of all smooth maps between the manifolds is the {\it Whitney $C^{\infty}$ topology}. For related systematic theory on singularities of differentiable maps, see \cite{golubitskyguillemin} for example. For fold maps, see as pioneering studies \cite{thom, whitney} on so-called generic maps into ${\mathbb{R}}^2$ on smooth manifolds whose dimensions are greater than or equal to $2$. \cite{saeki} is one of pioneering studies on fold maps and (differential) topological properties of manifolds admitting such maps.

\begin{Fact}[E.g. \cite{kobayashisaeki, shiota}]
Let $m > n \geq 1$ in Definition \ref{def:5}. For a fold map $f$ there, $W_f$ is a polyhedron.
In the case $n=2$, for a stable fold map $f$, $W_f$ is a simple polyhedron. In the case $n=2$, for a stable fold map $f$ such that $q_f {\mid}_{S(f)}$ is injective, or equivalently, a so-called {\rm simple} fold map $f$, $W_f$ is a simple polyhedron without vertices or a {\rm branched surface} in \cite{kitazawa5}.
\end{Fact}

The present paper concerns the following problems.

\begin{MainProb}
\label{mprob:1}
Is a given simple polyhedron homeomorphic to the Reeb space of a stable fold map on a closed manifold whose dimension is at least $3$ into a surface? 
\end{MainProb}
\begin{MainProb}
	\label{mprob:2}
	What can we say about the global topology of such a polyhedron? Which $3$-dimensional closed and connected manifolds can we embed this into, for example? This explicit problem is motivated by \cite{matsuzakiozawa, munozozawa, ozawa} for example. These studies are regarded as variants of ones on embeddability of graphs into surfaces.
\end{MainProb}

\cite{kitazawa5} gives some answers and related results. \cite{kitazawa5} also explains about existing related studies and the present study shows a kind of new explicit development closely related to this. Moreover, \cite{kitazawa5} has precise information more than the present paper
on history on these studies and facts related to our present study.
Main results of \cite{kitazawa5} are for simple fold maps into surfaces and simple polyhedra without vertices or equivalently, branched surfaces. 

The organization of our paper is as follows. The next section is devoted to presenting main results or Main Theorems \ref{mthm:1}, \ref{mthm:2}, \ref{mthm:3} and \ref{mthm:4}, starting from introducing notions and fundamental important
arguments. Except Main Theorem \ref{mthm:4}, these results are extensions of results presented as Main Theorems of \cite{kitazawa5} for simple polyhera. Main Theorem \ref{mthm:4} contains a result of a new type, presenting arguments of a new type. This is on non-orientable surfaces embedded as subpolyhedra of simple polyhedra. The third section is devoted to proofs of our Main Theorems and examples for example.

\section{Preliminaries and Main Theorems.}
$S^k:=\{x \in {\mathbb{R}}^{k+1} \mid ||x||=1\}$ denotes the $k$-dimensional unit sphere for an integer $k \geq 0$. This is a $k$-dimensional closed smooth submanifold in ${\mathbb{R}}^{k+1}$ with no boundary. $D^k:=\{x \in {\mathbb{R}}^{k} \mid ||x||\leq 1\}$ denotes the $k$-dimensional unit disk for an integer $k \geq 1$ and this is a $k$-dimensional compact and smooth submanifold.

\begin{Def}[E. g. \cite{kitazawa, saekisuzuoka, suzuoka}]
\label{def:7}
A stable fold map in Definition \ref{def:5} on an $m$-dimensional closed and orientable manifold into an $n$-dimensional manifold with no boundary satisfying $m>n \geq 1$ is said to be {\it standard-spherical} if the following conditions are satisfied.
\begin{enumerate}
\item The index of each singular point is always $0$ or $1$.
\item Preimages containing no singular points are disjoint unions of copies of the ($m-n$)-dimensional unit sphere $S^{m-n}$.
\end{enumerate}
\end{Def}

\begin{Def}
\label{def:8}
Let $m \geq 3$ be an integer. 
\begin{enumerate}
	\item A {\it normal} simple polyhedron is a simple polyhedron where the bundles over ${\sqcup}_j C_{i_j}$ whose fibers are PL homeomorphic to $K$ in Definition \ref{def:3} are always trivial. An {\it SSN} fold map is a standard-spherical fold map on an $m$-dimensional closed manifold into a surface with no boundary such that the Reeb space is normal. 
\item 
A map {\it born from an SSN fold map} $c:P \rightarrow N_c$ is a continuous {\rm (}piesewise smooth{\rm )} map such that there exist an SSN fold map $f:M \rightarrow N$ and a pair $(\Phi,\phi)$ of piesewise smooth homeomorphisms satisfying $\bar{f} \circ \Phi=\phi \circ c$.
\end{enumerate}
\end{Def}

The following proposition is a kind of fundamental principles, we can also know from referred articles on fold maps. 

\begin{Prop}
\label{prop:2}
Let $m \geq 3$ be an integer. Suppose that a continuous {\rm (}piesewise{\rm )} map $c:P \rightarrow N$ is locally born from an SSN fold map. Suppose also that the composition of the restriction of some premap for the branch ${\sqcup}_{i \in I} C_i$ with the restriction of the map $c_0$ to the branch is a smooth immersion having normal crossings only as crossings. Then there exist an $m$-dimensional closed manifold $M$, an SSN fold map $f:M \rightarrow N$ and a PL homeomorphism $\Phi:P \rightarrow W_f$ satisfying $\bar{f} \circ \Phi=c$ and $S(f)={\sqcup}_{i \in I} C_i$. Furthermore, $c$ is born from an SSN fold map.
\end{Prop}
\begin{proof}
We have presented a proof of this for the case where $P$ is a branched surface in \cite{kitazawa5} for example. Construction around a vertex is a most important ingredient in our proof.

Note that \cite{kitazawa2, kitazawa3, kitazawa4} also present explicit construction of local smooth maps around a vertex for example. However, we present a proof in a different form where the main ingredient is essentially same. We abuse the notation in Definition \ref{def:3}.

We obtain a local desired map onto $N(C_i)$ in Definition \ref{def:3} (\ref{def:3.3.1}). We construct the product map of a Morse function on the ($m-1$)-dimensional unit disk $D^{m-1}$ represented by the form of a function mapping $(x_1,\cdots,x_{m-1})$ to ${\Sigma}_{j=1}^{m-1} {x_j}^2+c_i$ for some real number $c_i$ and the identity map on $C_i$ and consider the quotient map onto the Reeb space, regarded as $C_i \times [0,1]$. We consider a suitable piesewise smooth homeomorphism onto $N(C_i)$ regarded as the product map of a piesewise smooth homeomorphism on $C_i$ and one on the closed interval $[0,1]$ where we regard $N(C_i)$ as $C_i \times [0,1]$ in the canonical way.

We obtain a local desired map onto $N(C_{i,j})$ in Definition \ref{def:3} (\ref{def:3.3.2}). We construct the product map of a Morse function on a manifold obtained by removing the interiors of three copies of the ($m-1$)-dimensional unit disk $D^{m-1}$ smoothly and disjointly embedded  in a copy of the ($m-1$)-dimensional unit sphere $S^{m-1}$ and the identity map on $C_{i,j}$ and consider the quotient map onto the Reeb space, regarded as $C_{i,j} \times K$. We consider a suitable piesewise smooth homeomorphism onto $N(C_{i,j})$, regarded as the product map of a piesewise smooth homeomorphism on $C_{i,j}$ and one on $K$ where we regard $N(C_{i,j})$ as $C_{i,j} \times K$ in the canonical way. The Morse function is as follows.
\begin{itemize}
\item The preimage of the minimal value and one or two of the connected components of the boundary coincide.
\item The preimage of the maximal value and the union of the remaining connected components of the boundary coincide.
\item There exists exactly one singular point and it is in the interior.
\end{itemize}
By a fundamental property of Morse functions, for each $p_a \in C_i$ we can remove the interior of the total space of a trivial smooth bundle over $[l_{a,1},u_{a,1}] \times (K_0 \bigcup K_{b,1})$ whose fiber is diffeomorphic to the ($m-2$)-dimensional unit disk $D^{m-2}$ apart from the singular set of the original product map. We can also remove that of a trivial smooth bundle over $[l_{a,2},u_{a,2}] \times (K_0 \bigcup K_{b,2})$ whose fiber is diffeomorphic to the ($m-2$)-dimensional unit disk $D^{m-2}$ apart from the singular set of the original product map. After the removal, we can identify the two new cornered smooth manifolds, diffeomorphic to $\partial D^{m-2} \times [-1,1] \times [-1,1]$, by a suitable diffeomorphism in such a way that we have a desired local map onto $N(C_i)$.  

We construct a trivial smooth bundle whose fiber is diffeomorphic to $S^{m-2}$ over the complementary set of the interior of a small regular neighborhood of the branch.

By gluing the local maps in a suitable way, we have a desired map onto $P$. By composing the map with $c$, we have a desired map.
This completes the proof.
\end{proof}

\begin{Def}
\label{def:9}
Hereafter, a map $c$ in Proposition \ref{prop:2} is defined as a map {\it born from an SSN fold map}.
\end{Def}

For a simple polyhedron $P$, a surface $N$ with no boundary, and a map $c:P \rightarrow N$ born from an SNS fold map, $B(c)$ denotes the branch $\sqcup C_j$ in Definition \ref{def:3}.


\begin{Def}
\label{def:10}
Let $\{c_j:X_j \rightarrow Y\}_{j \in J}$ be a family of finitely many piesewise smooth maps where $c_j$ is a map from a closed manifold $X_j$ into a manifold $Y$ with no boundary. If there exists a family $\{{\Phi}_j:Y \times [0,1] \rightarrow Y\}_{j \in J}$ of smooth isotopies satisfying ${\Phi}_j(y,0)=y$ for any $y \in Y$, then the family $\{c_j:X_j \rightarrow Y\}_{j \in J}$ is said to be {\it smoothly isotopic} to $\{{\Phi}_{1,j} \circ c_j:X_j \rightarrow Y\}_{j \in J}$ where ${\Phi}_{1,j}$ is a diffeomorphism on $Y$ mapping $y \in Y$ to ${\Phi}_{j}(y,1)$.
\end{Def}
We present four Main Theorems and except Main Theorem \ref{mthm:4}, these results are extensions of results presented as Main Theorems of \cite{kitazawa5}. Main Theorem \ref{mthm:4} is a theorem of a new type and presenting arguments of a new type. This discusses subpolyhedra which are non-orinetable surfaces in branched surfaces or simple polyhedra. 
\begin{MainThm}
\label{mthm:1}
Let $c:P \rightarrow N$ be a map born from an SSN fold map.
Suppose that $\{T_j\}_{j \in J}$ is a family of finitely many circles which are disjointly embedded in $P$ as subpolyhedra. We also assume the following conditions.
\begin{itemize}
\item Let there exist a premap for the union of $B(c)$ and the disjoint union ${\sqcup}_{j \in J} T_j$ of all circles in the family $\{T_j\}_{j \in J}$. The composition of this with the restriction of $c$ to the union is a smooth immersion having normal crossings only as crossings.
\item ${\bigcup}_{j \in J} c(T_j)$ is image of the restriction of some smooth immersion of some compact and connected surface $S_C$ into $N$ to the boundary.
\end{itemize}
Then by attaching a surface {\rm (}PL{\rm )} homeomorphic to $S_C$ along ${\sqcup}_{j \in J} T_j$ on the boundary by a piesewise smooth homeomorphism, we have a new normal branched surface $P^{\prime}$ and a map $c^{\prime}:P^{\prime} \rightarrow N$ born from an SSN fold map.
\end{MainThm}

\begin{MainThm}
\label{mthm:2}
Let $c_0:P \rightarrow N$ be a map born from an SSN fold map.
Let $\{T_{0,j}\}_{j \in J}$ be a family of finitely many circles which are disjointly embedded in $P$ and regarded as subpolyhedra of $P$. Assume also the following conditions. 
\begin{itemize}
\item $N$ is connected and $N-c_0(P)$ is not empty.
\item There exists a family $\{D_{0,j}\}_{j \in J}$ of finitely many copies of the $2$-dimensional unit disk $D^2$ smoothly and disjointly embedded in N as sufficiently small polyhedra and $c_0(T_{0,j}) \subset {\rm Int}\ D_{0,j}$.
\item There exists a premap for the union of $B(c_0)$ and the disjoint union $({\sqcup}_{j \in J} T_{0,j})$.
The composition of this premap with the restriction of $c_0$ to the previous union is a smooth immersion having normal crossings only as crossings.
\item The family of the piesewise smooth maps in the family $\{c_0 {\mid}_{T_{0,j}}\}_{j \in J}$ is smoothly isotopic to a family of smooth immersions having normal crossings only as crossings such that the restriction of some smooth immersion of some compact and connected surface $S_C$ into $N$ to the boundary is the disjoint union of all immersions of this family.
\end{itemize} 
Then we have the following two.
\begin{enumerate}
\item We have a new map $c:P \rightarrow N$ by a suitable piesewise smooth homotopy $F_c$ from $c_0$ to $c$.
\item We can apply Main Theorem \ref{mthm:1} to $c$ by setting $T_j:=F_c(T_{0,j} \times \{1\})$ and choosing $S_C$ as a compact and connected.
\end{enumerate}
\end{MainThm} 

We need some additional notions for our remaining Main Theorems.

\begin{Def}
	\label{def:11}
	A continuous (piesewise smooth) map between polyhedra is said to be {\it regarded as embeddings locally} (resp. {\it piesewise smooth embeddings locally}) if at each point in the polyhedron of the domain there exists a regular regular neighborhood of it the restriction of the map to which is an embedding (resp. a piesewise smooth embedding).
\end{Def}
Related to this defintion, we only consider piesewise smooth cases.

The following definition is for cases we encounter in our paper.

\begin{Def}
	\label{def:12}
	Let $c:P \rightarrow N$ be a map born from an SSN fold map.
	For a copy $D_c$ of the $2$-dimensional unit disk $D^2$ embedded as a subpolyhedron of $P$, suppose the following conditions.
	\begin{enumerate}
		\item There exists a premap for the union of $B(c)$ and the boundary $\partial D_c$ of $D_c$ such that the composition of the premap with the restriction of $c$ to this union is a smooth immersion having normal crossings only as crossings.
		\item The restriction $c {\mid}_{D_c}$ is regarded as piesewise smooth embeddings locally.

	\end{enumerate}
	We call such an embedded disk $D_c$ a {\it weakly and normally embedded disk with respect to $c$}. Suppose also that the intersection $P_{R,{\rm N}} \bigcap D_c$ of any non-orientable connected component $P_{R,{\rm N}}$ of $P-B(c)$ and $D_c$ is the empty set. Then, we call such an embedded disk $D_c$ a {\it normally embedded disk with respect to $c$}.
\end{Def}
For example, in Main Theorem \ref{mthm:2}, if each $D_j$ is in $P-B(c_0)$ and in an connected component of $P-B(c_0)$ which may not be orientable (which is orientable), then this is a (resp. weakly and) normally embedded disk with respect to $c_0$. See \cite{kitazawa5} for the case $D_j$ is in $P-B(c_0)$ with a branched surface $P$.

In Main Theorem \ref{mthm:3}, we also need the notion of the {\it Heegaard genus} of a $3$-dimensional closed and connected manifold. We can decompose such a manifold into two copies of a so-called {\it handlebody} along a closed and connected surface, which is called a {\it Heegaard surface}. A {\it handlebody} is a $3$-dimensional PL manifold obtained by attaching finitely many copies of $D^1 \times D^2$ one after another by piesewise smooth homeomorphisms each of which is from $\partial D^1 \times D^2=S^0 \times D^2$ onto the union of a pair of copies in the family of finitely many copies of the $2$-dimensional unit disk $D^2$ smoothly and disjointly embedded in the boundary of a copy of the $3$-dimensional unit disk $D^3$. The number of the copies of $D^1 \times D^2$ is well-defined and this is called the {\it genus} of the handlebody. The {\it Heegaard genus} of the closed and connected manifold is the minimal number of the genera of the handlebodies we obtain via the decompositions before or the minimal number of genera of the Heegaard surfaces. 
See \cite{hempel} for related theory on $3$-dimensional manifolds.

\begin{MainThm}
\label{mthm:3}
Let $c_0:P \rightarrow N$ be a map born from an SSN fold map.
Let $\{T_j\}_{j \in J}$ be a family of $l>0$ circles which are disjointly embedded in $P$ and regarded as subpolyhedra of $P$. 

Furthermore, assume also the following conditions.
\begin{enumerate}[{\rm (C}1{\rm )}]
\item
\label{thm:3C1}
 There exists a premap for the union of $B(c_0)$ and the disjoint union of all circles in the family $\{T_j\}_{j \in J}$. The composition of the premap with the restriction of $c_0$ to this union is a smooth immersion having normal crossings only as crossings.
\item
\label{thm:3C2}
 $N$ is connected and $N-c_0(P)$ is not empty.
\item 
\label{thm:3C3}
$T_j$ is regarded as the boundary of a copy $D_j$ of the $2$-dimensional unit disk $D^2$ regarded as a subpolyhedron of $P$.
\item 
\label{thm:3C4}
Distinct disks in $\{D_j\}_{j \in J}$ are mutually disjoint in $P$.
\item
\label{thm:3C5}
 Each disk $D_j$ is a weakly and normally embedded disk with respect to $c_0$.
\item
\label{thm:3C6}
 The family of the piesewise smooth maps in the family $\{c_0 {\mid}_{T_{0,j}}\}_{j \in J}$ is smoothly isotopic to a family of smooth immersions having normal crossings only as crossings such that the restriction of some smooth immersion of some compact, connected and orientable surface $S_C$ of genus $0$ into $N$ to the boundary is the disjoint union of all immersions of this family.
\item
\label{thm:3C7}
 We can embed $P$ in a $3$-dimensional closed and connected manifold $X_g$ whose Heegaard genus is $g$ where we argue in the PL category.
\end{enumerate}
Then we can obtain the new map $c^{\prime}$ from $c$ {\rm (}$c_0${\rm )} as in Main Theorems \ref{mthm:1} and \ref{mthm:2} by applying the methods of our proofs and a new normal simple polyhedron $P^{\prime}$. Moreover, we can apply Main Theorem \ref{mthm:2} here by setting $S_C$ as a compact, connected and orientable surface of genus $0$.

In addition, for each $D_j$, assume also that $D_j \bigcap B(c_0)$ is empty or a closed interval whose boundary is embedded into the boundary $\partial D_j$ and whose interior is embedded into the interior by a piesewise smooth embedding. In this case, we can also embed $P^{\prime}$ in an arbitrary $3$-dimensional closed, connected and orientable manifold represented as a connected sum of $X_g$ and $l-1$ manifolds each of which is $S^2 \times S^1$ or the total space of a non-trivial smooth bundle over $S^1$ whose fiber is diffeomorphic to $S^2$ where we discuss in the PL category. In addition, the Heegaard genus of the resulting manifold is $g+l-1$.

\end{MainThm}

In Main Theorem \ref{mthm:4}, we need {\it graphs}.
A {\it graph} is a $1$-dimensional simplicial complex. The {\it vertex set} of the graph is the set of all $0$-simplexes where each element of the set is a {\it vertex}. 
Note that the notion of a vertex of a graph is different from that of a vertex of a simple polyhedron, defined in Definition \ref{def:3}.
The {\it edge set} of it is the set of all $1$-simplexes where each element of the set is an {\it edge}.

A {\it subgraph} of a graph means a subcomplex of the graph.

For a graph $G$, a {\it path} $p:[0,1] \rightarrow G$ from a vertex $v_1$ to another vertex $v_2$ means a piesewise smooth map satisfying the following conditions.
\begin{itemize}
	\item $p(t)=v_{t+1}$ for $t=0,1$.
	\item We have a sequence $\{t_j\}_{j=1}^{l+1} \subset [0,1]$ of length $l+1>1$ 
	\begin{itemize}
		\item $t_1=0$ and $t_{l+1}=1$.
		\item $t_j<t_{j+1}$ for any integer $1 \leq j \leq l$.
		\item The preimage of the vertex set is the set of all real numbers in the sequence.
		\item Each closed interval $[t_j,t_{j+1}]$ is mapped onto some edge of the graph by a piesewise smooth homeomorphism.
	\end{itemize}
\end{itemize}
	
	 We have a sequence of $\{(p(t_j),p([t_{j},t_{j+1}]))\}_{j=1}^l$. 
This is a sequence each element of which is a pair $(p(t_j),p([t_{j},t_{j+1}]))$ of a vertex and an edge containing the vertex $p(t_j)$ and another vertex $p(t_{j+1})$.
This sequence is also defined as a {\it path}.

As an equivalent way, we can define the graph as an abstract simplicial complex in a natural way. 
In short, vertices and edges are defined as abstract objects of a certain class. We can define the notions before similarly. We also adopt this way in discussing our graphs. Definition \ref{def:13} gives an example.
 
 Our graphs have no {\it loops} by the definition.
 Our graphs may be so-called {\it multigraphs} by the definition. In other words between two vertices, there may exist more than $1$ edges.
 
 We omit other elementary notions and properties on graphs.

\begin{Def}
\label{def:13}
In Definition \ref{def:12}, the graph $G_{c,D_c}$ satisfying the following
conditions is said to be the {\it graph associated with the pair} $(c,D_c)$.

	\begin{enumerate}
		\item
		The vertex set of $G_{c,D_c}$ is the set of all connected components of $P-B(c)$ whose intersections with $D_c$ are not empty {\rm (}and as a result the intersections are surfaces{\rm )}.
		\item
		We define the set of all edges connecting two distinct vertices or connected components $P_{R,1}$ and $P_{R,2}$ of $P-B(c)$ as the set consisting of all connected components of $\overline{P_{R,1}} \bigcap \overline{P_{R,2}} \bigcap (B(c)-\{p_a\}_{a \in A})$ such that the intersections with the disk $D_c$ are not empty  where $\{p_a\}_{a \in A} \subset B(c_0)$ denotes the set of all vertices of the simple polyhedron $P$. Note that these connected components are connected components of the $1$-dimensional smooth manifold $B(c)-\{p_a\}_{a \in A}$. Let each edge be denoted by the pair of the connected component of $\overline{P_{R,1}} \bigcap \overline{P_{R,2}} \bigcap (B(c)-\{p_a\}_{a \in A})$ and the unordered pair of the two vertices.
	
		\end{enumerate}
\end{Def}
\begin{MainThm}
	
\label{mthm:4}
Let $c_0:P \rightarrow N$ be a map born from an SSN fold map.
Let $\{T_j\}_{j \in J}$ be a family of $l>0$ circles which are disjointly embedded in $P$ and regarded as subpolyhedra of $P$.
We suppose what follows.
\begin{enumerate}[{\rm (C}1{\rm )}]
\item
\label{thm:4C1}
 There exists a premap for the union of $B(c_0)$ and the disjoint union of all circles in the family $\{T_j\}_{j \in J}$. The composition of the premap with the restriction of $c_0$ to this union is a smooth immersion having normal crossings only as crossings.
\item
\label{thm:4C2}
 $N$ is connected and $N-c_0(P)$ is not empty.
\item 
\label{thm:4C3}
$T_j$ is regarded as the boundary of a copy $D_j$ of the $2$-dimensional unit disk $D^2$ regarded as a normally embedded disk with respect to $c_0$.
\item
\label{thm:4C4}
 Distinct disks in $\{D_j\}_{j \in J}$ are mutually disjoint in $P$.

\item 
\label{thm:4C5}
We have a normally embedded disk $D_0$ with respect to $c_0$ such that
each graph $G_{c_0,D_j}$ associated with the pair $(c_0,D_j)$ is a subgraph of the graph $G_{c_0,D_0}$ associated with the pair $(c_0,D_0)$.
\item 
\label{thm:4C6}
The family of the piesewise smooth maps in the family $\{c_0 {\mid}_{T_{0,j}}\}_{j \in J}$ is smoothly isotopic to a family of immersions having normal crossings only as crossings such that the restriction of some smooth immersion of some compact, connected and orientable surface $S_C$ into $N$ to the boundary is the disjoint union of all immersions of this family.
\end{enumerate}
In this case, we can obtain the new map $c^{\prime}$ from $c$ {\rm (}$c_0${\rm )} as in Main Theorems \ref{mthm:1} and \ref{mthm:2} and a new normal simple polyhedron $P^{\prime}$ enjoying the following two properties in addition.
\begin{enumerate}
\item
\label{mthm:4.1}
 There exists a subpolyhedron PL homeomorphic to a non-orientable compact surface.
\item
\label{mthm:4.2}
In adddtion, we suppose the following conditions.
\begin{enumerate}
	\item
	\label{mthm:4.2.1}
	 ${\sqcup}_{j \in J} D_j$ is regarded as a subpolyhedron of a closed and connected surface $S_P$ which is also a subpolyhedron of $P$.
	\item We have the following graphs.
	\label{mthm:4.2.2}
	\begin{enumerate}
		\item \label{mthm:4.2.2.1}
		 From $S_P$, we have a graph $G_{S_p}$ where its vertex set and its edge set are as follows.
		\begin{enumerate}
			\item The vertex set consists of all connected components of $P-B(c_0)$ whose intersections with $S_P$ are not empty.   
			\item The edge set consists of connected components of the $1$-dimensional smooth manifold $B(c_0)-\{p_a\}_{a \in A}$ where $\{p_a\}_{a \in A} \subset B(c_0)$ denotes the set of all vertices of the simple polyhedron $P$.
			\item We define the set of all edges connecting two distinct vertices or connected components $P_{R,1}$ and $P_{R,2}$ of $P-B(c_0)$ as the set consisting of all connected components of $\overline{P_{R,1}} \bigcap \overline{P_{R,2}} \bigcap (B(c_0)-\{p_a\}_{a \in A})$. Let each edge be denoted by the pair of the connected component of $\overline{P_{R,1}} \bigcap \overline{P_{R,2}} \bigcap (B(c)-\{p_a\}_{a \in A})$ and the unordered pair of the two vertices as in Definition \ref{def:13}.
		\end{enumerate}
		\item \label{mthm:4.2.2.2}
		 For some pair of graphs $G_{c,D_{j_1}}$ and $G_{c,D_{j_2}}$, there exists a path $p:[0,1] \rightarrow G_{c,D_0}$ from some vertex $v_1 \in G_{c,D_{j_1}}$ to one $v_2 \in G_{c,D_{j_2}}$ enjoying the following properties.
		\begin{enumerate}
			
			\item 
			\label{mthm:4.2.2.2.1}
			As in the definition of a graph, there exists a sequence $\{t_j\}_{j=1}^{l+1} \subset \mathbb{R}$ of length $l>0$ enjoying the following properties.
			\begin{itemize}
		\item $t_1=0$ and $t_{l+1}=1$. 
		\item $t_j<t_{j+1}$ for $1 \leq j \leq l$.
		\item The set of all elements in $\{t_j\}_{j=1}^{l+1}$ is the preimage of the vertex set of $G_{c,D_0}$ for the path $p$. Put $v_j:=p(t_j)$.
		\item There exist an even and positive integer $l^{\prime}$ and a subsequence $\{t_{i_j}\}_{j=1}^{l^{\prime}}$ of length $l^{\prime}$ satisfying $i_{j^{\prime}}<i_{j^{\prime}+1}$ for any $1 \leq j^{\prime} \leq l^{\prime}-1$ and $i_{l^{\prime}} \leq l$.
		\item There exists another subsequence $\{t_{{i^{\prime}}_j}\}_{j=1}^{l^{\prime}}$ of length $l^{\prime}$ satisfying ${i}_{j^{\prime}}<{i^{\prime}}_{j^{\prime}} \leq {i}_{j^{\prime}+1}$ for any $1 \leq j^{\prime} \leq l^{\prime}-1$ and 
		${i}_{l^{\prime}}<{i^{\prime}}_{l^{\prime}} \leq l+1$.
        \end{itemize} 
			\item
			\label{mthm:4.2.2.2.2}
			 Furthermore, for the sequence before, the following properties are enjoyed.
			\begin{itemize}
				\item Let $j$ be an integer satisfying
				 $1 \leq j<i_1$, ${i^{\prime}}_{l^{\prime}} \leq j \leq l$ or
				  ${i^{\prime}}_{j^{\prime}} \leq j<i_{j^{\prime}+1}$ for some integer $1 \leq j^{\prime} \leq l^{\prime}-1$. In this case, $e_j:=p([t_j,t_{j+1}])$ is an edge in the 
				 graph $G_{S_P}$.
				\item For any integer $1 \leq j^{\prime} \leq l^{\prime}$ and two vertices $v_{t_{i_{j^{\prime}}}}$ and $v_{t_{{i^{\prime}}_{j^{\prime}}}}$, there exists at least one edge connecting these two in the graph $G_{S_P}$ and for any normally embedded disk $D_{c_0}$ with respect to $c_0$, there exist no edges connecting these two in the graph $G_{c_0,D_{c_0}}$ associated with $(c_0,D_{c_0})$.
			\end{itemize}
		\end{enumerate}
		
	\end{enumerate}
	\end{enumerate}
Then as a result, there exists a closed, connected and non-orientable surface which is also a subpolyhedron of $P^{\prime}$ and $P^{\prime}$ cannot be embedded into $S^3$ or a $3$-dimensional closed, connected and orientable manifold whose homology group is isomorphic to that of $S^3$ as a subpolyhedron where the coefficient ring is the group $\mathbb{Z}/2\mathbb{Z}$ of order $2$.
\end{enumerate}

\end{MainThm}
\section{Proofs of Main Theorems.}
We prove Main theorems. They can be shown by using similar arguments in proofs of Main Theorems of \cite{kitazawa5} with new additional arguments. However, we do not need to understand these original proofs well to understand our proofs.
\begin{proof}[A proof of Main Theorem \ref{mthm:1}.]
As the proof of Main Theorem 1 of \cite{kitazawa5}, we can attach a surface PL homeomorphic to $S_C$ along ${\sqcup}_{j \in J} T_j$ via a piesewise smooth homeomorphism between the boundaries and obtain a new continuous (piesewise smooth) map $c^{\prime}$. We need to remember the definitions of a normal simple polyhedron (Definition \ref{def:3}) and a map born from an SSN fold map (Definition \ref{def:9}) and we can easily construct a desired map by preserving the property that the maps are born from SSN fold maps. 
\end{proof}
\begin{Rem}
In \cite{kitazawa5}, as Main Theorem \ref{mthm:1}, we have relations between homology groups, cohomology groups and rings and fundamental groups of branched surfaces $P$ and $P^{\prime}$. We can extend the arguments in a similar way. Rigorous arguments and results are left to readers.
\end{Rem}
\begin{proof}[A proof of Main Theorem \ref{mthm:2}.]
We prove this by using an argument similar to that of the proof of Main Theorem 2 of \cite{kitazawa5}. 

We explain about related arguments precisely. We take a point $p_0$ and its small neighborhood $D_{p_0}$ which is a smooth compact submanifold in ${\rm Int}\ (N-c_0(P)) \subset N-c_0(P)$ and diffeomorphic to the $2$-dimensional unit disk $D^2$. We have a family $\{D_{1,j} \subset {\rm Int}\ D_{0,j}\}_{j \in J}$ of finitely many copies of the $2$-dimensional unit disk $D^2$ smoothly embedded in $N$ satisfying $c_0(T_{0,j}) \subset {\rm Int}\ D_{1,j}$. 

We can take (the image of) a smooth curve $t_{p_0,j}$ connecting $p_0$ and a point in the interior of $D_{1,j}$ which is a smooth embedding in such a way that for distinct $j=j_1,j_2$ the intersection of the images are the one-point set $\{p_0\}$ and we do so.
By a suitable piesewise homotopy from $c_0$, we can have a map $c_1$ born from an SSN fold map. Furtheremore, we may assume that the following properties are enjoyed. An exposition of this type is also presented in the proof of Main Theorem 2 of \cite{kitazawa5}.

\begin{itemize}
\item There exists a sufficiently small regular neighborhood $N(t_{p_0,j})$ of the image of $t_{p_0,j}$.
\item The connected component of ${c_0}^{-1}(D_{1,j} \bigcup N(t_{p_0,j}))$ containing $T_{0,j}$ as a subpolyhedron, denoted by $C_{{c_0}^{-1}, D_{0,j}, T_{0,j}}$, is in ${c_1}^{-1}({\rm Int}\ D_{p_0})$, for each $j$.
\item ${c_1}^{-1}({\rm Int}\ D_{p_0})$ is in the union ${\bigcup}_{j \in J} {c_0}^{-1}(N^{\prime}(t_{p_0,j}))$
where a suitable small regular neighborhood $N^{\prime}(t_{p_0,j})$ of $D_{1,j} \bigcup N(t_{p_0,j})$ is chosen for each $j$: ${\rm Int}\ N^{\prime}(t_{p_0,j}) \supset D_{1,j} \bigcup N(t_{p_0,j})$ holds of course.
\item The family  $\{c_1 {\mid}_{T_{0,j}}\}_{j \in J}$ of the restrictions is smoothly isotopic to the family $\{c_0 {\mid}_{T_{0,j}}\}_{j \in J}$ of the restrictions.
\item The piesewise smooth homotopy is regarded as a map $F_{c_0,c_1}$ such that on some small regular neighborfood $N(C_{{c_0}^{-1},D_{0,j}, T_{0,j}})$ of $C_{{c_0}^{-1},D_{0,j}, T_{0,j}}$
 there exists a smooth isotopy ${\Phi}_{j,c_0,c_1}:Y \times [0,1] \rightarrow Y$ making the relation $F_{c_0,c_1}(x,t)={\Phi}_{j,c_0,c_1}(c_0(x),t)$. Furthermore, in the deformation by this smooth homotopy, at points outside the union ${\bigcup}_{j \in J} N(C_{{c_0}^{-1},D_{0,j}, T_{0,j}})$ of the small regular neighborhoods, the values are fixed.
\end{itemize}
A new important ingredient is, to respect Definition \ref{def:10} and the condition on the image $c_0(T_{0,j}) \subset {\rm Int}\ D_{0,j}$. These assumptions and the definition also allow us to change the map $c_1$ by a piesewise smooth homotopy from $c_1$ to $c$ suitably to apply Main Theorem \ref{mthm:1} in the last. Note that the existence of the premap for the union of $B(c)$ and the disjoint union ${\sqcup}_{j \in J} T_j$ giving a smooth immersion having normal crossings only as crossings is due to arguments on so-called {\it generic} properties, discussed as fundamental arguments on singularity theory of differentiable maps and differential topology for example. This is essential in the definition of a stable map, presented shortly in the first section. See \cite{golubitskyguillemin} again. In \cite{kitazawa5}, this is discussed shortly. 
We can also deform the map to obtain $c$ only around $D_{p_0}$ and ${c_1}^{-1}(D_{p_0})$ by Definition \ref{def:10} with the condition $c_0(T_{0,j}) \subset {\rm Int}\ D_{0,j}$. In the deformation by a pisewise smooth homotopy from $c_1$ to $c$, at points outside the union ${\bigcup}_{j \in J} N(C_{{c_0}^{-1},D_{0,j}, T_{0,j}})$, the values are fixed as before.
This completes the proof.
\end{proof}
\begin{proof}[A proof of Main Theorems \ref{mthm:3}.]
We can prove Main Theorem \ref{mthm:3} as a specific case of Main Theorems \ref{mthm:1} and \ref{mthm:2} in the present paper. 

We can prove the former part or the application of (a suitably revised version of) Main Theorem \ref{mthm:2}, presented before, in a style similar to that of our proof of our Main Theorem \ref{mthm:2}. Note that $T_j$ is assumed to be the boundary of some copy of the $2$-dimensional unit disk embedded as a subpolyhedron by ({\rm C}\ref{thm:3C3}). We also need ({\rm  C}\ref{thm:3C1}), ({\rm C}\ref{thm:3C2}) and ({\rm C}\ref{thm:3C6}). 

The remaining part is shown in a style similar to that of the proof of Main Theorem 2 of \cite{kitazawa5} and we give its proof.


We explain about the resulting polyhedron $P^{\prime}$ and a new $3$-dimensional closed and connected manifold where we can embed this as a subpolyhedron. 



By removing the interiors of the $l$ disjoint subpolyhedra PL homeomorphic to the $3$-dimensional unit disk $D^3$ from a copy of the $3$-dimensional unit sphere $S^3$ (with the canonical PL structure), we have a manifold ${S^3}_{(l)}$.
We can see that the surface $S_C$ can be embedded in ${S^3}_{(l)}$ as a subpolyhedron enjoying the following properties.
\begin{itemize}
	\item The boundary $\partial S_C$ is embedded in the boundary $\partial {S^3}_{(l)}$. 
	\item Distinct connected components of the boundary $\partial S_C$ are embedded in distinct connected components of the boundary $\partial {S^3}_{(l)}$.
	\item The interior ${\rm Int}\ S_C$ is embedded in the interior ${\rm Int}\ {S^3}_{(l)}$.
\end{itemize}

We explain about the new desired $3$-dimensional manifold where $P^{\prime}$ can be embedded as a subpolyhedron. We can choose small $l$ disjoint subpolyhedra PL homeomorphic to the $3$-dimensional unit disk $D^3$ with the canonically defined PL structure in $X_g$. We take such a family, denoted by $\{{D^3}_j\}_{j \in J}$.
Furthermore, we can take this enjoying the following properties and we take this respecting the properties. 
Each $D_j$ is assumed to be a normally embedded disk with respect to $c_0$ and disks in the given family $\{D_j\}_{j \in J}$ are mutually disjoint. 

\begin{itemize}
	\item The boundary $c_0(\partial D_j)$ is embedded into the boundary $\partial {D^3}_j$ via a piesewise smooth embedding.
	\item The interior $c_0({\rm Int}\ D_j)$ is embedded into the interior ${\rm Int}\ {D^3}_j$ via a piesewise smooth embedding.
\end{itemize}
For each $D_j$, $D_j \bigcap B(c_0)$ is assumed to be empty or a closed interval whose boundary is embedded into the boundary $\partial D_j$ and whose interior is embedded into the interior by a piesewise smooth embedding. Due to this assumption, via a small piesewise smooth isotopy we can deform the original embedding of $P$ obeying the following conditions where we do not know about (suitable) extensions of this to general cases.
\begin{itemize}
	\item In deforming the original embedding, we fix all points in the complementary set of the disjoint union ${\sqcup}_{j \in J} {\rm Int}\ D_j$ of the interiors of the disks in $\{D_j\}_{j \in J}$.
	\item After the deformation, the disjoint union ${\sqcup}_{j \in J} {\rm Int}\ D_j$ of the interiors of the disks in $\{D_j\}_{j \in J}$ is moved outside the disjoint union ${\sqcup}_{j \in J} {D^3}_j$ of the copies of the disk in $\{{D^3}_j\}_{j \in J}$.
\end{itemize}
See also FIGURE \ref{fig:0.1}.
\begin{figure}
	\includegraphics[width=50mm]{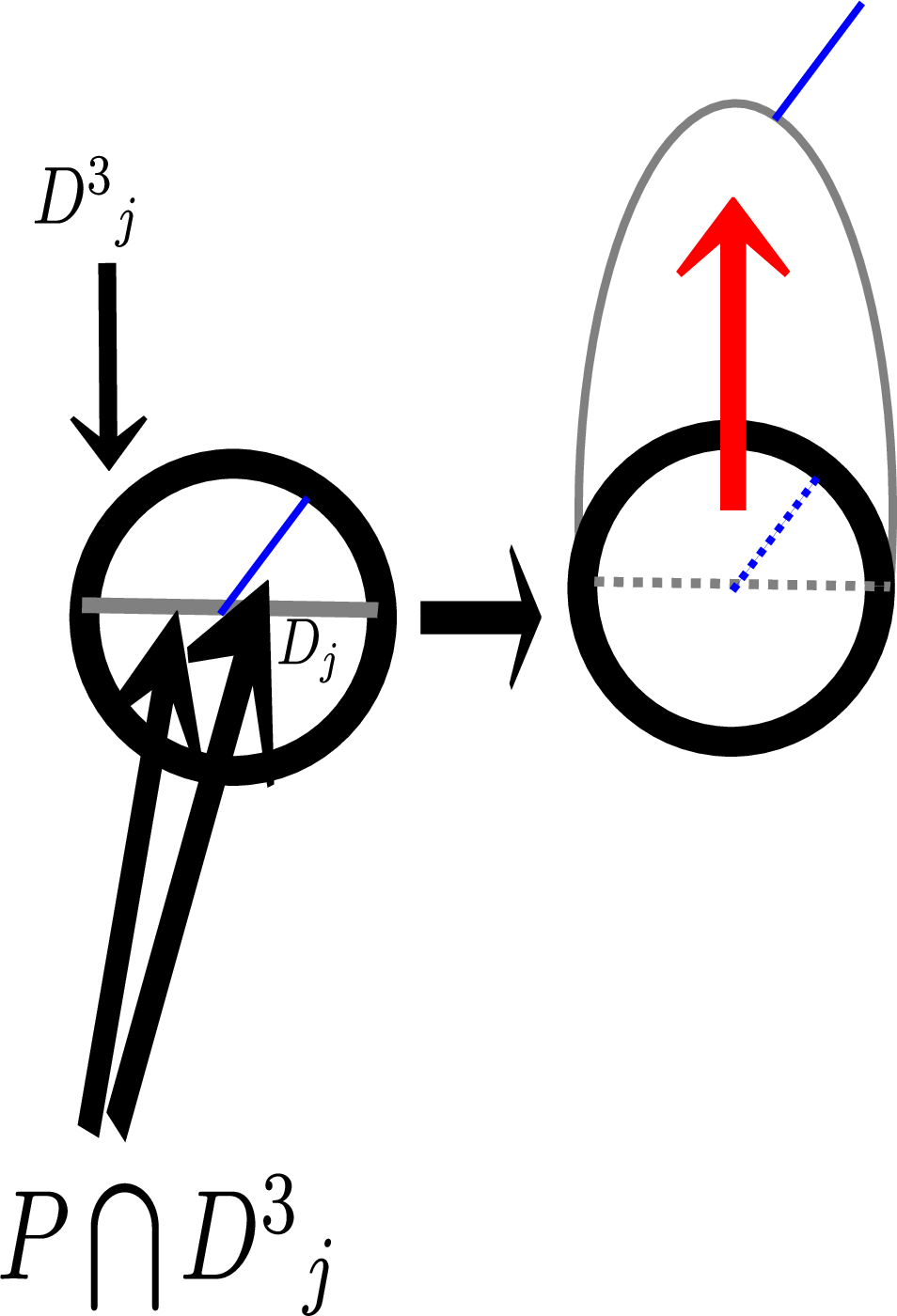}
	\caption{The case $D_j \bigcap B(c_0)$ is assumed to be a closed interval whose boundary is embedded into the boundary $\partial D_j$ and whose interior is embedded into the interior by a piesewise smooth embedding: $D_j \subset {D^3}_j$ is represented by the gray segment and $P \bigcap {D^3}_j \subset {D^3}_j$ is represented by the union of the gray segment and the blue segment. The deformation of the original embedding by a piesewise smooth isotopy: the red arrow shows this.}
	\label{fig:0.1}
\end{figure}
We remove the interiors of the $3$-dimensional subpolyhedra of $X_g$ in the family $\{{D^3}_j\}_{j \in J}$. After that we attach ${S^3}_{(l)}$ by a piesewise smooth homeomorphism along the boundaries such that the restriction to the boundary $\partial S_C$ is a piesewise smooth homeomorphism onto the subpolyhedron of $P$ to obtain $P^{\prime}$. 

 The resulting $3$-dimensional manifold can be also regarded as a manifold (PL) homeomorphic and diffeomorphic to the manifold $X_{g+l-1}$ obtained by a finite iteration of the following procedures starting from $X_g$.

\begin{itemize}
	\item Choose small two disjoint subpolyhedra PL homeomorphic to the $3$-dimensional unit disk $D^3$ (with the canonically defined PL structure) in the $3$-dimensional closed and connected manifold $X_{g^{\prime}}$.
	\item Attach a manifold PL homeomorphic $S^2 \times D^1$ to the boundary via piesewise smooth homeomorphism between the boundaries in such a way that the resulting manifold is a $3$-dimensional closed and connected manifold $X_{g^{\prime}+1}$. It can be represented as a connected sum of $X_{g^{\prime}}$ and a copy of $S^1 \times S^2$ or the total space of a non-trivial smooth bundle over $S^1$ whose fiber is diffeomorphic to $S^2$. We can obtain both manifolds. This is due to 
	what follows.
	 On two copies of the $2$-dimensional unit sphere $S^2$, consider a circle embedded as a subpolyhedron in each copy. Suppose that a circle is mapped onto another circle by a piesewise smooth homeomorphism. Then this is extended to both orientation preserving and reversing piesewise smooth homeomorphisms between the given copies of $S^2$ where arbitrary orientations are given.    
\end{itemize}

By these arguments, we can see that $P^{\prime}$ can be embedded in the resulting manifold as a subpolyhedron.

This completes the proof.
\end{proof}

\begin{proof}[A proof of Main Theorem \ref{mthm:4}]
We prove (\ref{mthm:4.1}) first.

From the map $c_0 {\mid}_{D_{0}}$ where $D_{0}$ is a disk giving the graph $G_{0}$ such that all our graphs in $\{G_j\}_{j \in J}$ are subgraphs, we obtain the suitable orientation for the connected component of $P-B(c_0)$ represented by each vertex of the graph $G_{0}$. Moreover, these orientations are canonically induced from the map $c_0 {\mid}_{D_{i_0}}$ and a suitable orientation of $N$. We can also orient the disk $D_j$ canonically by the definitions and properties of the disks $D_j$ and $D_{0}$ and the graphs $G_j$ and $G_0$. 


By the definitions and properties of the disks $D_j$ and $D_{0}$ and the graphs $G_j$ and $G_0$, we can connect $D_{0}$ and each $D_j$ by a smooth curve $t_{D_j}:[0,1] \rightarrow P-B(c_0)$ which is an embedding and maps $(0,1)$ to $P-(B(c_0) \bigcup ({\sqcup}_{j \in J} D_j) \bigcup D_0)$ for $j \in J$ where $J$ is the set defined in (\ref{mthm:4.2}) after presenting (\ref{mthm:4.1}). We can take these curves in such a way that the images of distinct curves are always disjoint. As a small regular neighborhood of the union of the image of the disjoint union ${\sqcup}_{j \in J} t_{D_j}$ of the curves and $D_0 \bigcup {\sqcup}_{j \in J} D_j$, we have a new $2$-dimensional manifold $N(D)$ which is PL homeomorphic to the $2$-dimensional unit disk. We can see that we can induce the orientation of $N(D)$ from $N$ canonically and that this also gives the orientations of each $D_j$ and $D_0$ given before canonically.


Here we may regard that each $T_j$, regareded as the boundary of the disk $D_j$, is mapped into a sufficiently small copy of the $2$-dimensional unit disk $D^2$ smoothly embedded in $N$ by $c_0$ as in the situation of Main Theorem \ref{mthm:2}. We may also regard that
the disjoint union ${\sqcup}_{j \in J} T_j$ is mapped into another sufficiently small copy of the $2$-dimensional unit disk $D^2$ smoothly embedded in $N$ by $c_0$.
This with the condition (C\ref{thm:4C6}) on smoothly isotopic families allows us to apply the method of the proof of Main Theorem \ref{mthm:2} which is revised suitably, for example. We apply the construction of the maps via piesewise smooth homotopies in the proofs of our Main Theorem \ref{mthm:2} (, Main Theorem \ref{mthm:1}, or \ref{mthm:3} in specific cases,) and fundamental arguments on the orientations of the surfaces.
$N(D)$ is, after the new compact, connected and orientable surface $S_C$ in ({\rm C}\ref{thm:4C6}) is attached in our situation, regarded as a subpolyhedron containing a compact and non-orientable surface as a subpolyhedron. 
Such a non-orientable surface is obtained by removing the interiors of the disks in the family $\{D_j\}_{j \in J}$ from this subpolyhedron, for example. This completes the proof of (\ref{mthm:4.1}).

We prove (\ref{mthm:4.2}).

We orient $N$ as before. Hereafter, we may assume that $S_P$ is orientable to prove this.
We can orient $S_P$ in such a way that the orientation of the connected component $v_1$ of $P-B(c_0)$, which is also a subset of $S_P$, is the orientation induced canonically from $N(D)$ and $N$ as in the beginning.
For $j \leq {i}_{1}$, the orientation of $v_j$ induced canonically from $N(D)$ and $N$ and the orientation induced canonically from the oriented surface $S_P$ are same.

We compare the orientation of the connected component $v_{i^{\prime}_1}$ of $P-B(c_0)$ induced canonically from $N(D)$ and $N$ to the orientation induced canonically from the oriented surface $S_P$. Note that  the connected component $v_{i^{\prime}_1}$ of $P-B(c_0)$ is also a subset of $S_P$. These orientations are mutually distinct, mainly due to the assumption on the non-existence of the graph $G_{c_0,D_{c}}$ associated with $(c_0,D_{c})$ for any normally embedded disk $D_{c_0}$ with respect to $c_0$, which is in the property (\ref{mthm:4.2.2.2.2}). See also FIGURE \ref{fig:0.2}.

\begin{figure}
	\includegraphics[width=50mm]{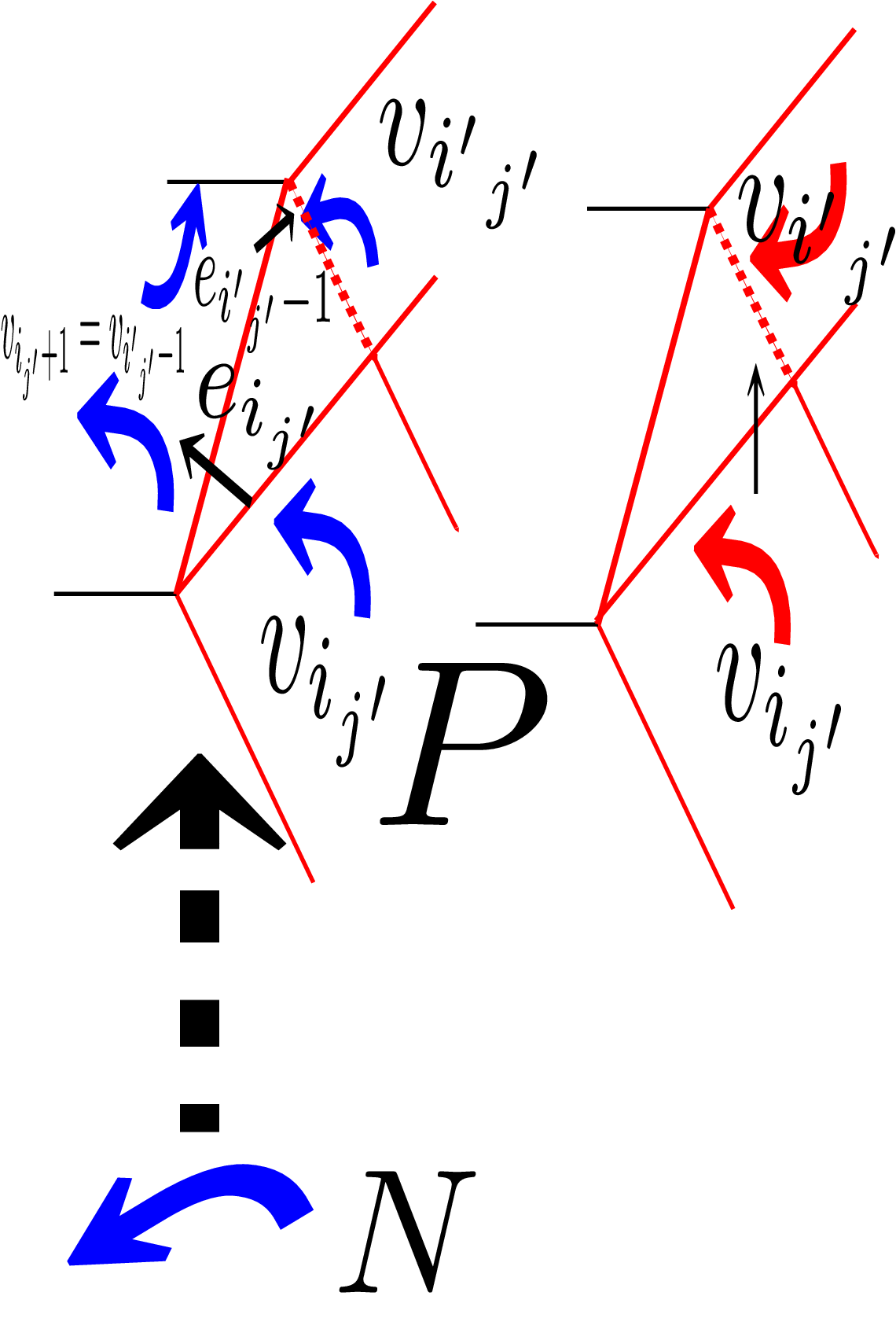}
	\caption{The left figure is for orientations of each region or connected component of $P-B(c_0)$ (, $P-B(c)$, or $P-B(c^{\prime})$) and each disk in the family $\{D_j\}_{j \in J}$ induced from $D_0$, $N(D)$ and $N$. Here $N$ is oriented (suitably). The right figure is for orientations of each region and each disk in the family $\{D_j\}_{j \in J}$ induced from a surface $S_P$, which is oriented (suitably). In these two figures, regions surrounded by the red segments are subpolyhedra of the surface $S_P$. For an edge, the notation $e_j:=p_j([t_j,t_{j+1}])$ is used as in the property (\ref{mthm:4.2.2.2.2}).}
	\label{fig:0.2}
\end{figure}

For $i^{\prime}_1 \leq j \leq {i}_{2}$, the orientation of $v_j$ induced canonically from $N(D)$ and $N$ and the orientation induced canonically from $S_P$ are distinct.

The orientation of the connected component $v_{i^{\prime}_2}$ of $P-B(c_0)$ induced canonically from $N(D)$ and $N$ is same as the orientation induced canonically from $S_P$. Note that the connected component $v_{i^{\prime}_2}$ of $P-B(c_0)$ is also a subset of $S_P$. See also FIGURE \ref{fig:0.2} again.

$l^{\prime}$ is even and positive. 
We apply a kind of inductions. By fundamental arguments on orientations of surfaces, which are keys in the proof of (\ref{mthm:4.1}), we have a desired closed adn connected non-orientable surface which is also a subpolyhedron of $P^{\prime}$. This surface is obtained by attaching the surface $S_C$ to $S_P$ and removing the interiors of all disks in the family ${\{D_j\}}_{j \in J} \subset S_P$.

This completes the proof of (\ref{mthm:4.2}).

This completes the proof.
\end{proof}



We present an example for Main Theorem \ref{mthm:2} or \ref{mthm:3}.

\begin{Ex}
	\label{ex:1}
	In Main Theorem \ref{mthm:2}, if $c_0 {\mid}_{T_{0,j}}$ is regarded as the restriction of the map $c_0 {\mid}_{D_j}$ on a copy $D_j$ of the $2$-dimensional unit disk $D^2$ embedded as a subpolyhedron in $P$, all disks in $\{D_j\}_{j \in J}$ are disjointly embedded, and the restriction $c_0 {\mid}_{{\sqcup}_{j \in J} D_j}$ is regarded as a piesewise smooth embedding, then we can apply Main Theorem \ref{mthm:2} by taking $S_C$ as a compact, connected and orientable surface of genus $0$. If each $D_j$ is in some connected component of $P-B(c_0)$, then we can apply Main Theorem \ref{mthm:3}. 
	\end{Ex}

We also present an example for Main Theorem \ref{mthm:4}. Remember that Proposition \ref{prop:2} gives a fundamental tool in constructing SSN fold maps, for example.

Here we also regard ${\mathbb{R}}^k$ as the natural vector space. 
Each point is identified with a vector canonically.
For two elements $x_1,x_2 \in {\mathbb{R}}^k$, we consider the vector $x_1-x_2 \in {\mathbb{R}}^k$ defined by considering the difference.

\begin{Ex}
\label{ex:2}
We can construct an SSN fold map $f:M \rightarrow {\mathbb{R}}^2$ enjoying the following properties on a suitable $m$-dimensional closed and connected manifold $M$ into ${\mathbb{R}}^2$ for any $m>2$.
\begin{enumerate}
\item $f {\mid}_{S(f)}$ is an embedding.
\item $f(S(f))=\{x \in {\mathbb{R}}^2 \mid ||x||=1,2,8,9,10,11\}$.
\item The index of each singular point in the preimage of $\{x \in {\mathbb{R}}^2 \mid ||x||=l\}$ is always $0$ for $l=9,11$ and $1$ for $l=1,2,8,10$.
\item The number of connected components of the preimage of $\{x \in {\mathbb{R}}^2 \mid l<||x||<l+1\}$ is $11-l$ for $l=8,9,10$, $4$ for $l=0$ and $5$ for $l=1$. The number of connected components of the preimage of $\{x \in {\mathbb{R}}^2 \mid 2<||x||<8\}$ is $4$.
\end{enumerate}

See FIGURE \ref{fig:1}. This is a so-called {\it round} fold map, defined first in \cite{kitazawa0.1,kitazawa0.2,kitazawa0.3} by the author. See also \cite{kitazawa0.4,kitazawa0.5}. We can construct this such that the Reeb space is as in FIGURE \ref{fig:3}. We construct the fold map so that the Reeb space can be embedded into ${\mathbb{R}}^3$ and that there exists a PL embedding $F$ satisfying the following properties.
\begin{enumerate}
\item $F(W_f)$ is the union of the following sets.
\begin{enumerate}
\item $\{x \in {\mathbb{R}}^3 \mid ||x||=8,10\}$.
\item $\{(r\cos u,r\sin u,0) \in {\mathbb{R}}^3 \mid 8 \leq r \leq 9,10 \leq r \leq 11, u \in \mathbb{R}\}$.
\item All segments whose boundaries are of the form $\{(2\cos u,2\sin u,u_{\rm u}) \in S_{\rm u},(\cos u,\sin u,u_{\rm d}) \in S_{\rm d}\}$ for $u \in \mathbb{R}$ where $S_{\rm u}:=\{(x_1,x_2,x_3) \in {\mathbb{R}}^2 \times \mathbb{R} \mid ||(x_1,x_2,x_3)-(0,0,x_3)||=2, x_3>0\} \bigcap \{x \in {\mathbb{R}}^3 \mid ||x||=10\}$ and $S_{\rm d}:=\{(x_1,x_2,x_3) \in {\mathbb{R}}^2 \times \mathbb{R} \mid ||(x_1,x_2,x_3)-(0,0,x_3)||=1, x_3>0\} \bigcap \{x \in {\mathbb{R}}^3 \mid ||x||=8\}$.
\end{enumerate} 
\item $(F \circ q_f)(S(f))=\{x \in {\mathbb{R}}^2 \times \{0\} \mid ||x||=8,9,10,11\} \sqcup S_{\rm u} \sqcup S_{\rm d}$.
\end{enumerate}
FIGURE \ref{fig:3} presents the intersection of the image of the Reeb space and $\mathbb{R} \times \{0\} \times \mathbb{R}$.
We can construct a new fold map $f^{\prime}$ enjoying the following properties on a suitable $m$-dimensional closed and connected manifold $M^{\prime}$ by removing the preimage of $\{x \in {\mathbb{R}}^2 \mid ||x-(-5,0)|| \leq \frac{5}{2}.\}$ and gluing a suitable new smooth map instead. This map is presented in FIGURE \ref{fig:2}. 
\begin{enumerate}
\item $f^{\prime} {\mid}_{S(f^{\prime})}$ is an embedding.
\item $f^{\prime}(S(f^{\prime}))=\{x \in {\mathbb{R}}^2 \mid ||x||=1,2,8,9,10,11\} \sqcup \{x \in {\mathbb{R}}^2 \mid ||x-(-5,0)||=1,2\}$.
\item The index of each singular point in the preimage of $\{x \in {\mathbb{R}}^2 \mid ||x||=l\}$ is always $0$ for $l=9,11$ and $1$ for $l=1,2,8,10$. The index of each singular point in the preimage of $\{x \in {\mathbb{R}}^2 \mid ||x-(-5,0)||=1,2\}$ is $1$.
\item The number of connected components of the preimage of $\{x \in {\mathbb{R}}^2 \mid l<||x||<l+1\}$ is $11-l$ for $l=8,9,10$, $4$ for $l=0$ and $5$ for $l=1$. The number of connected components of the preimage of each connected component of $\{x \in {\mathbb{R}}^2 \mid 2<||x||<8\}-\{x \in {\mathbb{R}}^2 \mid 1 \leq ||x-(-5,0)|| \leq 2\}$ is $4$. The
number of connected components of the preimage of $\{x \in {\mathbb{R}}^2 \mid 1<||x-(-5,0)||<2\}$ is $5$. 
\end{enumerate}
We can also construct $f^{\prime}$ such that the Reeb space $W_{f^{\prime}}$ can be embedded into ${\mathbb{R}}^3$ by a piesewise smooth embedding $F^{\prime}$ enjoying the following properties.
\begin{enumerate}
\item $F^{\prime}(W_{f^{\prime}})$ is the union of the following sets.
\begin{enumerate}
\item $\{x \in {\mathbb{R}}^3 \mid ||x||=8,10\}$.
\item $\{(r\cos u,r\sin u,0) \in {\mathbb{R}}^3 \mid 8 \leq r \leq 9,10 \leq r \leq 11, u \in \mathbb{R}\}$.
\item All segments whose boundaries are of the form $\{(2\cos u,2\sin u,u_{\rm u}) \in S_{\rm u},(-5+\cos u,\sin u,u_{\rm d}) \in S_{\rm d}\}$ for $u \in \mathbb{R}$.
\item All segments whose boundaries are of the form $\{(-5+2\cos u^{\prime},2\sin u^{\prime},{u^{\prime}}_{\rm u}) \in {S^{\prime}}_{\rm u},(-5+\cos u^{\prime},\sin u^{\prime},{u^{\prime}}_{\rm d}) \in {S^{\prime}}_{\rm d}\}$ for $u^{\prime} \in \mathbb{R}$ where ${S^{\prime}}_{\rm u}:=\{(x_1,x_2,x_3) \in {\mathbb{R}}^2 \times \mathbb{R} \mid ||(x_1,x_2,x_3)-(-5,0,x_3)||=2, x_3>0\} \bigcap \{x \in {\mathbb{R}}^3 \mid ||x||=8\}$ and ${S^{\prime}}_{\rm d}:=\{(x_1,x_2,x_3) \in {\mathbb{R}}^2 \times \mathbb{R} \mid ||(x_1,x_2,x_3)-(-5,0,x_3)||=1, x_3<0\} \bigcap \{x \in {\mathbb{R}}^3 \mid ||x||=10\}$.
\end{enumerate}
\item $(F^{\prime} \circ q_{f^{\prime}})(S(f^{\prime}))=\{x \in {\mathbb{R}}^2 \times \{0\} \mid ||x||=8,9,10,11\} \sqcup S_{\rm u} \sqcup S_{\rm d} \sqcup {S^{\prime}}_{\rm u} \sqcup {S^{\prime}}_{\rm d}$. 
\end{enumerate}
We can say that $W_{f^{\prime}}$ and $q_{f^{\prime}}$ are obtained by Main Theorem \ref{mthm:1} from $W_f$ and $q_f$. $W_{f^{\prime}}$ contains a closed, connected and non-orientable surface as a subpolyhedron. In FIGURE \ref{fig:3} we can see this by considering the union of the $2$-dimensional surface $S_P$ represented by gray curves and segments and the newly attached copy of $S^1 \times [-1,1]$ and removing the interiors of two suitable disks embedded as subpolyhedra here. The resulting closed surface is (PL) homeomorphic to the Klein Bottle. 

We can also apply Main Theorem \ref{mthm:4} for the two disjointly embedded disks here. 
We can define these disks as
${D^{\prime}}_{\rm u}:=\{(x_1,x_2,x_3) \in {\mathbb{R}}^2 \times \mathbb{R} \mid ||(x_1,x_2,x_3)-(-5,0,x_3)|| \leq 2, x_3>0\} \bigcap \{x \in {\mathbb{R}}^3 \mid ||x||=8\}$ and ${D^{\prime}}_{\rm d}:=\{(x_1,x_2,x_3) \in {\mathbb{R}}^2 \times \mathbb{R} \mid ||(x_1,x_2,x_3)-(-5,0,x_3)|| \leq 1, x_3<0\} \bigcap \{x \in {\mathbb{R}}^3 \mid ||x||=10\}$.
We have graphs in Main Theorem \ref{mthm:4} (\ref{mthm:4.2}). These disks are normally embedded disks with respect to the map $\bar{f}:W_f \rightarrow {\mathbb{R}}^2$ and we set these disks as $D_{j_1}$ and $D_{j_2}$. The graph associated to the pair of the map $\bar{f}$ and each disk here is a graph consisting of exactly one vertex. We have a normally embedded disk $D_0$ with respect to the map $\bar{f}:W_f \rightarrow {\mathbb{R}}^2$ such that these two graphs are subgraphs and we have a situation of Main Theorem \ref{mthm:4} (\ref{mthm:4.2}) for $l^{\prime}=2$. We also have a case for $l^{\prime}=4$.
\end{Ex}
\begin{center}
\begin{figure}
\includegraphics[width=35mm]{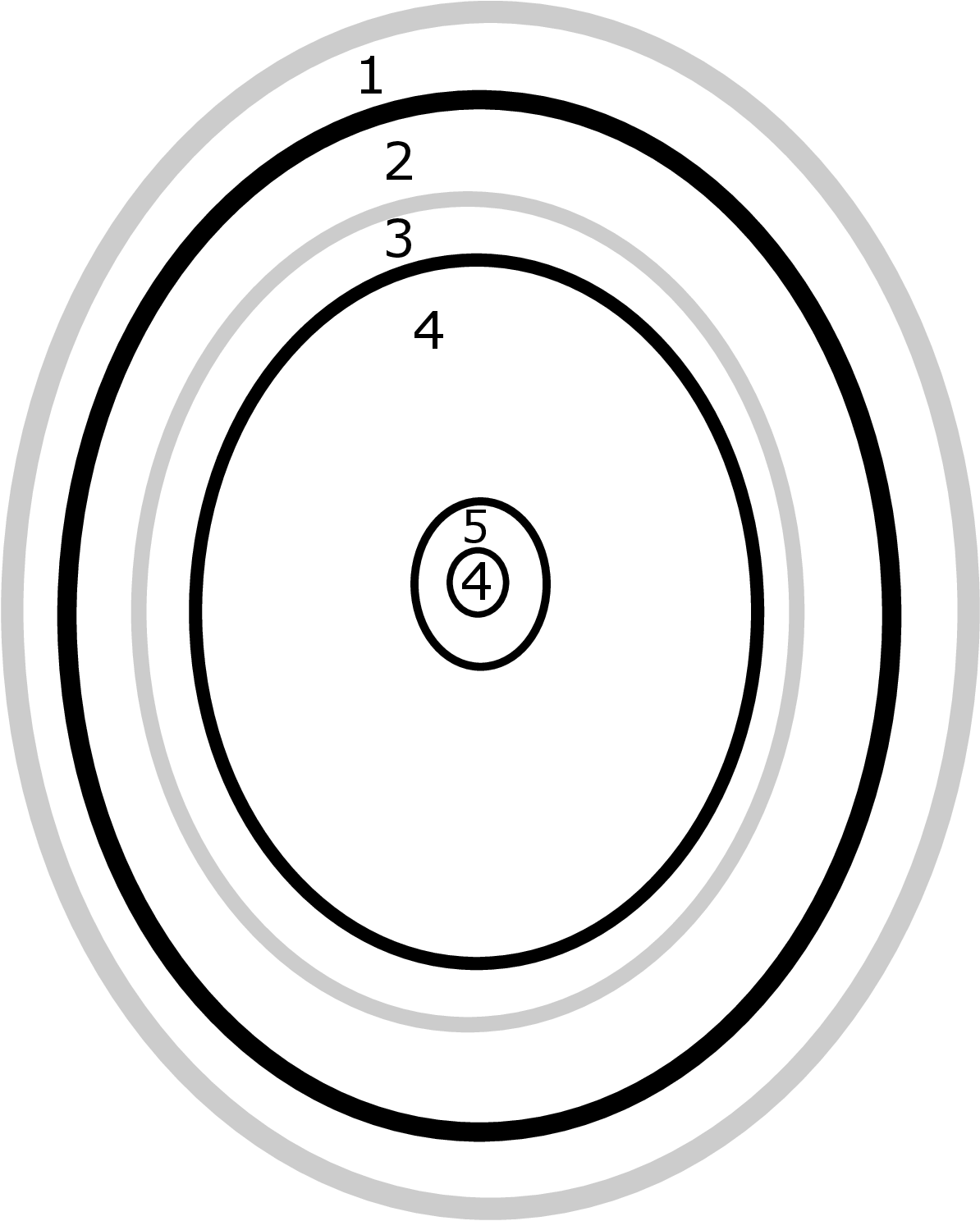}
\caption{The image of $f$. Circles are for the singular value set. Numbers stand for the numbers of connected components of the preimages for points in the regular value set. For connected components represented by black circles, the indices of singular points in the preimages are always $0$ and for ones represented by gray circles the indices of singular points in the preimages are always $1$.}
\label{fig:1}
\end{figure}
\begin{figure}
\includegraphics[width=35mm]{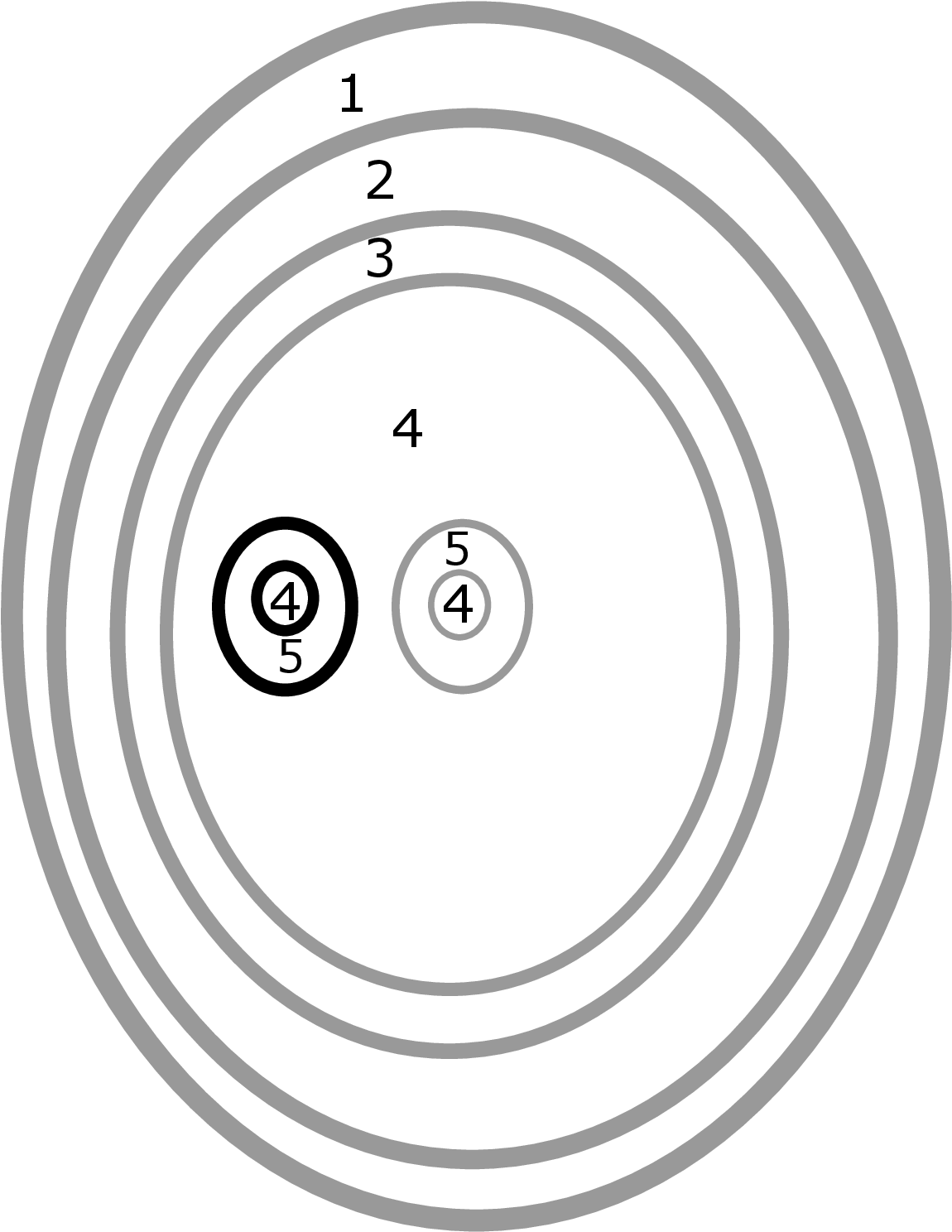}
\caption{The image of $f^{\prime}$. Circles are for the singular value set. Numbers stand for the numbers of connected components of the preimages for points in the regular value set. For connected components represented by black circles, the indices of points in the preimages are always $0$ and for ones represented by gray circles they are $1$.}
\label{fig:2}
\end{figure}
\begin{figure}
\includegraphics[width=60mm]{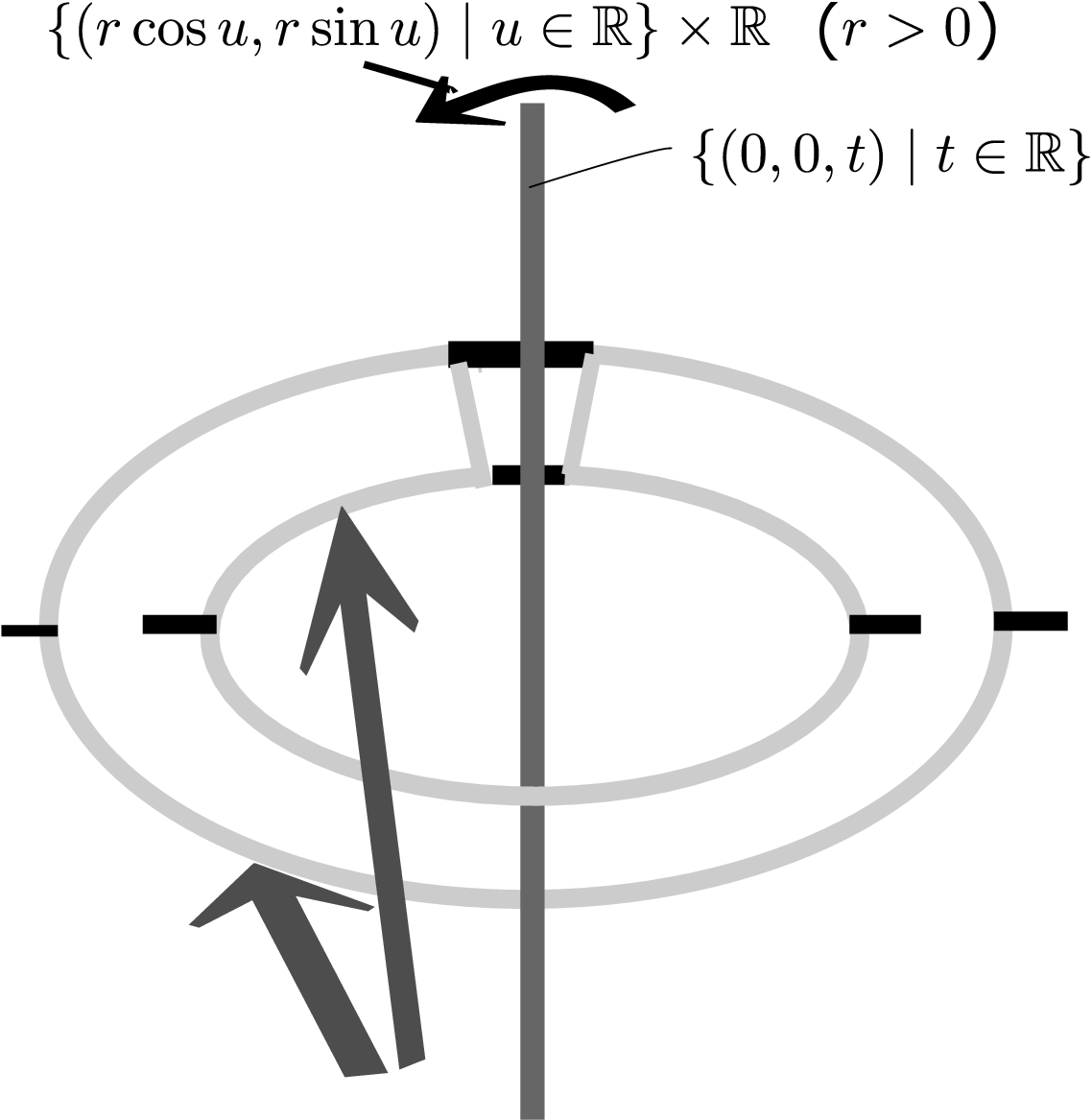}
\caption{The images of $F$ and $F^{\prime}$. A manifold PL homeomorphic to $S^1 \times [-1,1]$ is attached to the boundaries of two disks located around the two points indicated by arrows by a suitable PL homeomorphism.}
\label{fig:3}
\end{figure}
\end{center}
\section{Acknowledgment and data availability.}
The author is a member of JSPS KAKENHI Grant Number JP17H06128 "Innovative research of geometric topology and singularities of differentiable mappings" (Principal Investigator: Osamu Saeki) and the present work is supported by this. 

The present study is also related to a joint research project at Institute of Mathematics for Industry, Kyushu University (20200027), ''Geometric and constructive studies of higher dimensional manifolds and applications to higher dimensional data'', principal investigator of which is the author. The author would like to thank people supporting our new research project.
This is a kind of projects applying geometric theory on higher dimensional differentiable manifolds developing through studies of the author to higher dimensional data analysis and visualizations.


The author declares that data essentially supporting our present study are all in the present paper. 

\end{document}